\documentclass[11pt]{article}

\usepackage[latin1]{inputenc}
\usepackage{t1enc}
\usepackage{amsmath, amssymb, amsfonts, amsthm, amsopn,epsfig}
\usepackage[none]{hyphenat}
\usepackage{tikz}
\usepackage{float}
\usepackage{graphicx}
\usepackage{caption}
\usepackage[margin=1in]{geometry}
\usepackage[toc,page]{appendix}
\usepackage{epstopdf}
\usepackage{etoolbox}
\usepackage[colorlinks=true]{hyperref}
\usepackage{dsfont}
\hypersetup{
	colorlinks=true,
	linkcolor=blue,
	filecolor=magenta,      
	urlcolor=cyan,
	citecolor=blue
}
\usepackage{cleveref}

\theoremstyle{plain}
\newtheorem{theorem}{Theorem}[section]

\newtheorem{claim}[theorem]{Claim}
\newtheorem{lemma}[theorem]{Lemma}

\theoremstyle{definition}
\newtheorem{definition}{Definition}

\newcommand{\cF}{\mathcal{F}}
\newcommand{\cB}{\mathcal{B}}
\newcommand{\C}{\mathcal{C}}
\newcommand{\cE}{\mathcal{E}}
\newcommand{\cS}{\mathcal{S}}
\renewcommand{\Re}{{\mathbb R}}
\newcommand{\Red}{\Re^d}
\DeclareMathOperator{\vol}{vol}
\DeclareMathOperator{\diam}{diam}

\title{The Quantitative Fractional Helly theorem}
\author{N\'ora Frankl\thanks{The Open University and HUN-REN Alfr\'ed R\'enyi Institute of
Mathematics, \emph{e-mail}: \textbf{nora.frankl@open.ac.uk}. Partially supported by ERC Advanced Grant 'GeoScape'.}, Attila Jung\thanks{E\"otv\"os University and HUN-REN Alfr\'ed R\'enyi Institute of
Mathematics, \emph{e-mail}: \textbf{jungattila@gmail.com}. Research supported by the ERC Advanced Grant ``ERMiD'' and by the Thematic Excellence Program TKP2021-NKTA-62 of the National Research, Development and Innovation Office.}, Istv\'an Tomon\thanks{Ume\r{a} University, \emph{e-mail}: \textbf{istvan.tomon@umu.se}. Research supported in part by the Swedish Research Council grant VR 2023-03375.}}
\date{}

\begin{document}
\sloppy

\maketitle

\begin{abstract}
    Two celebrated extensions of Helly's theorem are the \emph{Fractional Helly theorem} of Katchalski and Liu (1979) and the \emph{Quantitative Volume theorem} of B\'ar\'any, Katchalski, and Pach (1982). Improving on several recent works, we prove an optimal combination of these two results. We show that given a family $\cF$ of $n$ convex sets in $\mathbb{R}^d$ such that at least $\alpha \binom{n}{d+1}$ of the $(d+1)$-tuples of $\cF$ have an intersection of volume at least 1, then one can select $\Omega_{d,\alpha}(n)$ members of $\cF$ whose intersection has volume at least $\Omega_d(1)$.

 Furthermore, with the help of this theorem, we establish a quantitative version of the $(p,q)$ theorem of Alon and Kleitman. Let $p\geq q\geq d+1$ and let $\cF$ be a finite family of convex sets in $\mathbb{R}^d$ such that among any $p$ elements of $\cF$, there are $q$ that have an intersection of volume at least $1$.  Then, we prove that there exists a family $T$ of $O_{p,q}(1)$ ellipsoids of volume $\Omega_d(1)$ such that every member of $\cF$ contains at least one element of $T$.

 Finally, we present extensions about the diameter version of the Quantitative Helly theoerm.
\end{abstract}

\section{Introduction}

\emph{Helly's theorem} \cite{Helly} is one of the fundamental results of discrete geometry. Discovered in 1913, it states that if $\cF$ is a family of $n$ convex sets in $\mathbb{R}^d$ such that the intersection of any $d+1$ elements of $\cF$ is nonempty, then the intersection of all elements of $\cF$ is also nonempty. By considering a family of hyperplanes in general position, it is easy to see that the theorem is no longer true if $d+1$ is replaced with any smaller number. Helly's theorem inspired a large number of generalizations and extensions, see the recent surveys \cite{BK_Survey,HW}.

The \emph{Fractional Helly theorem} of Katchalski and Liu \cite{KL79} states that if one only assumes that an $\alpha$-fraction of the $(d+1)$-tuples of $\cF$ have a nonempty intersection, then one can select a $\beta$-fraction of $\cF$ with a nonempty intersection for some $\beta=\beta(d,\alpha)>0$. The optimal choice  $\beta=1-(1-\alpha)^{1/(d+1)}$ was determined by Kalai \cite{Kalai84}. The Fractional Helly theorem played an important role in the resolution of the celebrated Hadwiger-Debrunner $(p,q)$-problem \cite{HD57} by Alon and Kleitman \cite{AK92}, which we discuss in more detail later in this section.

Another important extension of Helly's theorem, called the \emph{Quantitative Volume theorem}, gives information about the volume of the intersection. This theorem, proved by B\'ar\'any, Katchalski, and Pach \cite{BKP}, states that if any $2d$ members of $\cF$ have an intersection of volume at least 1, then the intersection of the elements of $\cF$ has volume at least $v(d)$ for some $v(d)>0$ depending only on $d$. Brazitikos \cite{Braz} proved, slightly improving on a breakthrough result of Nasz\'odi \cite{Nasz16}, that the maximal volume $v(d)$ satisfying this property is between $(cd)^{-3d/2}$ and $(cd)^{-d/2}$ for some absolute constant $c>0$. Note that in the Quantitative Volume theorem, it is not enough to consider the intersections of $k$ element subsets of $\cF$ for any $k<2d$, as witnessed by the following construction. For any $\varepsilon>0$, consider a cube $C$ of volume $\varepsilon$ in $\mathbb{R}^d$, and let $\cF$ be the set of $2d$ halfspaces bounded by the face hyperplanes of $C$, each of which contains $C$. Then any $2d-1$ elements of $\cF$ have an intersection of infinite volume, while the intersection of all elements of $\cF$ is equal to $C$, thus has volume $\varepsilon$.

Recently, there were several attempts to establish a combination of the Fractional Helly theorem and the Quantitative Volume theorem, which one may call as the \emph{Quantitative Fractional Helly theorem} (or QFH for short). Sarkar, Xue, and Sober\'on \cite{SXS} proved that if $k=d(d+3)/2$ and $\cF$ is family of $n$ convex sets in $\mathbb{R}^d$ such that at least $\alpha$-fraction of the $k$-tuples  $(A_1,\dots,A_k)\in \cF^{(k)}$ satisfy $\vol(A_1\cap\dots\cap A_k)\geq 1$, then one can select at least $\beta$-fraction of the elements of $\cF$ such that their intersection has volume at least $v(d)$, where $\beta=\beta(d,\alpha)>0$ and $v(d)>0$. They showed that one can take $v(d)=d^{-d}$, and proposed the conjecture that a similar statement holds for $k=2d$ as well. In this direction, Jung and Nasz\'odi \cite{JN} showed that one can take $k=3d+1$ with $v(d)=d^{-O(d^2)}$, which was further improved by Jung \cite{Jung} to $k=3d$ with $v(d)=d^{-O(d^3)}$. Our first theorem proves the conjecture of Sarkar, Xue, and Sober\'on \cite{SXS}. 

\begin{theorem}\label{thm:QFH_2d}
    For every positive integer $d$, there exists $c_0=c_0(d)>0$ such that the following holds. Let $\cF$ be a family of $n$ convex sets in $\mathbb{R}^d$ such that at least $\alpha \binom{n}{2d}$ of the $(2d)$-tuples of members of $\cF$ have an intersection of volume at least $1$. Then there are at least $c_0\alpha n$ members of $\cF$ whose intersection has volume at least $d^{-c_1d}$ for some $c_1>0$.
\end{theorem}

While the Quantitative Volume theorem does not hold if one only considers the intersections of $k$-tuples for any $k<2d$, this does \emph{not} imply that the Quantitative Fractional Helly theorem cannot hold either with $k<2d$. Indeed, our next theorem goes beyond and shows that the optimal value is $k=d+1$. Observe that a family of thickened hyperplanes in general position shows that it is not possible to prove a QFH theorem with $k\leq d$.

\begin{theorem}\label{thm:QFH_d+1}
   For every positive integer $d$, there exists $v(d)>0$, and for every $\alpha\in (0,1)$, there exists $\beta=\beta(d,\alpha)>0$ such that the following holds. Let $\cF$ be a family of $n$ convex sets in $\mathbb{R}^d$ such that at least $\alpha \binom{n}{d+1}$ of the $(d+1)$-tuples of members of $\cF$ have an intersection of volume at least $1$. Then one can find at least $\beta n$ elements of  $\cF$ whose intersection has volume at least $v(d)$.
\end{theorem}

Comparing Theorems \ref{thm:QFH_2d} and \ref{thm:QFH_d+1}, the latter gives the best possible value of $k$. However, the former gives stronger quantitative bound for the volume $v(d)$.
While we will not try to optimize our constants in the proof of Theorem \ref{thm:QFH_d+1}, careful calculations show that we can take $v(d)=2^{-2^{...2}}$, where the tower is of height $O(d)$.

\subsection{A quantitative \texorpdfstring{$(p,d+1)$}{(p,d+1)}-theorem}

We also use Theorem~\ref{thm:QFH_d+1} to establish a quantitative analog of the celebrated $(p,q)$-theorem of Alon and Kleitman \cite{AK92}. A \emph{transversal} of a family of sets $\cF$ is a set of points $X$ such that every member of $\cF$ contains an element of $X$. Then the $(p,q)$-theorem states that for every pair of integers $(p,q)$ with $p\geq q\geq d+1$, there exists $r=r(p,q)$ such that the following holds. If $\cF$ is a family of convex sets in $\mathbb{R}^d$ such that among any $p$ members there exist $q$ with a nonempty intersection, then $\cF$ has a transversal of size~$r$. Observe that the $(p,d+1)$-theorem already implies the $(p,q)$-theorem for every $p\geq q\geq d+1$.

One can naturally extend the definition of transversals to incorporate volume. For $v>0$,  a family $T$ of convex bodies of volume $v$ is a \emph{quantitative $v$-transversal} of a family $\cF$ of convex sets in $\Red$, if every element of $\cF$ contains at least one element of $T$. Sarkar, Xue, and Sober\'on \cite{SXS} proved that if $p\geq d(d+3)/2=:k$, and $\cF$ is a family of convex sets in $\Red$ such that among any $p$ elements of $\cF$ there are $k$ with intersection of volume at least 1, then $\cF$ has a $v$-transversal of size at most $r=r(p)$ with $v=d^{-d}$. Jung and Nasz\'odi \cite{JN} showed that a similar result holds for $k=3d+1$, although with worse bounds on~$v$. Our theorem shows that one can take $k=d+1$, the smallest possible value of~$k$.

\begin{theorem}[Quantitative $(p,d+1)$ theorem]\label{thm:pqsmallp}
	For every positive integer $d$  there exists $v(d)>0$, and for every integer $p \geq d+1$, there exists an integer 	$r = r(p,d)$ such that the following holds. Let $\cF$ be a finite family of convex sets in $\Red$, each of volume at least one, such that among any $p$ members of $\cF$, there exist $d+1$ whose intersection has volume at least 1. Then $\cF$ has a quantitative $v(d)$-transversal of size at most $r$.
\end{theorem}

Alon, Kalai, Matou\v{s}ek, and Meshulam~\cite{Alon02} showed that fractional Helly-type results imply $(p,q)$-theorems in abstract convexity spaces (see \cite{Alon02} for precise definitions). However, their results are not directly applicable in the quantitative setting. Despite this, the ideas of their proofs can be adapted to deal with quantitative questions, as was done by Jung and Nasz\'odi in~\cite{JN}. We follow the same path to prove Theorem~\ref{thm:pqsmallp}. For a general treatment of combinatorial proofs in the quantitative Helly-type setting, see~\cite{Jung}.

\subsection{Diameter Helly theorems}

B\'ar\'any, Katchalski and Pach \cite{BKP} proved a counterpart of their Quantitative Volume theorem, where the sizes of the intersections are measured by their diameter instead of their volume. They showed that given a family $\mathcal{F}$ of convex bodies in $\mathbb{R}^d$, if the diameter of the intersection of any $2d$ elements is at least $1$, then the diameter of the intersection of all the convex bodies in the family is $\Omega_d(1)$. Also, this is sharp as $2d$ cannot be replaced by any smaller number. The following fractional variant was proved by Sober\'on \cite{Soberon}. If  $\cF$ is a family of $n$ convex sets in $\mathbb{R}^d$ such that at least $\alpha$-fraction of the $(2d)$-tuples of members of $\cF$ have an intersection of diameter at least $1$, then one can select at least $\beta n$ elements of  $\cF$ whose intersection has diameter at least $\Omega_d(1)$, for some $\beta=\beta(d,\alpha)>0$. 

As Travis Dillon (personal communication) pointed out, one can use our proof method to show the following improvement.

\begin{theorem}\label{thm:QFHdiameter_d+1}
For every positive integer $d$, there exists $c(d)>0$, and for every $\alpha\in (0,1)$, there exists $\beta=\beta(d,\alpha)>0$ such that the following holds. Let $\cF$ be a family of $n$ convex sets in $\mathbb{R}^d$ such that at least $\alpha \binom{n}{d+1}$ of the $(d+1)$-tuples of members of $\cF$ have an intersection of diameter at least $1$. Then one can find at least $\beta n$ elements of  $\cF$ whose intersection has diameter at least $c(d)$. 
\end{theorem}


Furthermore, an analog of our Theorem~\ref{thm:pqsmallp}, a $(p,d+1)$-theorem for the diameter can be also established.

\begin{theorem}\label{thm:diamter_pq}
	For every positive integer $d$  there exists $v(d)>0$, and for every integer $p \geq d+1$, there exists an integer 	$r = r(p,d)$ such that the following holds. Let $\cF$ be a finite family of convex sets in $\Red$, each of diameter at least one, such that among any $p$ members of $\cF$, there exist $d+1$ whose intersection has diameter at least 1. Then there exists a family $\mathcal{C}$ of $r$ convex sets of diameter at least $v(d)$ such that any member of $\cF$ contains at least one element of $\mathcal{C}$.
\end{theorem}

See Section \ref{sect:diameter} for a brief outline of the proofs.


\section{Preliminaries}

The proofs of our theorems rely on a combination of known results from convex geometry and extremal hypergraph theory, which we collect here for the reader's convenience. An important strengthening of Helly's theorem is the Colorful Helly theorem, first proved by Lov\'asz (see B\'ar\'any~\cite{Barany82}). 

\begin{lemma}[Colorful Helly theorem]\label{lemma:colorful}
Let $\cF_1,\dots,\cF_{d+1}$ be families of convex sets in $\mathbb{R}^d$ such that $A_1\cap\dots \cap A_{d+1}\neq \emptyset$ for every $(A_1,\dots,A_{d+1})\in \cF_1\times\dots\times \cF_{d+1}$. Then there exists $i\in [d+1]$ such that $\bigcap_{A\in \cF_i}A\neq \emptyset$.
\end{lemma}

We highlight that quantitative versions of the Colorful Helly theorem are also studied by Dam\'asdi, F\"oldv\'ari, and Nasz\'odi \cite{DFN}. They prove that if one has $3d$ families, and any colorful selection of $2d$ elements intersect in volume at least 1, then one of the families intersect in large volume. Unfortunately, this is not applicable in our case, as it is important for us to have exactly $d+1$ families. Connections between colorful and fractional Helly type theorems have been already known, Holmsen \cite{H} shows implications between the two in an abstract setting. 

We also apply the Quantitative Volume theorem as a black-box. The following sharp bounds on the volume are proved by Nasz\'odi \cite{Nasz16} and Brazitikos \cite{Braz}.

\begin{lemma}[Quantitative Helly theorem]\label{lemma:quantitative}
Let $\cF$ be a family of convex sets in $\mathbb{R}^d$ such that $$\vol(A_1\cap\dots \cap A_{2d})\geq 1$$ for every $A_1,\dots,A_{2d}\in \cF$. Then $\vol(\bigcap_{A\in \cF}A)\geq d^{-cd}$ for some absolute constant $c>0$.
\end{lemma} 

The following result tells us that any convex body $K$ in $\mathbb{R}^d$ is contained in a simplex of volume $O_d(\vol(K))$. Recall that a \emph{convex body} is a compact convex set with nonempty interior. While the next lemma gives the best known bound on the constant hidden in the $O_d(.)$ notation, any constant depending only on $d$ would serve our purposes. 

\begin{lemma}[Galicer, Merzbacher, and Pinasco \cite{GMP}]\label{lemma:smallest_simplex}
Let $K$ be a convex set in $\mathbb{R}^d$. Then there exists a simplex of volume at most $(cd)^{d/2}\cdot \vol(K)$ containing $K$ for some absolute constant $c>0$.
\end{lemma}

The next lemma provides another approximation of convex bodies with the help of ellipsoids. Its proof is an easy consequence of John's ellipsoid theorem \cite{John}.

\begin{lemma}\label{lemma:ellipsoid}
Let $K\subset \mathbb{R}^d$ be a convex body and let $E$ be the largest volume ellipsoid in $K$. If $E$ is centered at the origin, then $K\subset d\cdot E$. 
\end{lemma}

The following result is a simple exercise, and we omit its proof. In particular, we only use the trivial result that the number of regions $n$ hyperplanes cut $\mathbb{R}^d$ into is bounded by a function of $n$ and $d$.  

\begin{lemma}\label{lemma:hyperplane_cutting}
The maximum number of regions $n$ hyperplanes cut $\mathbb{R}^d$ into is $\sum_{k=0}^{d}\binom{n}{k}\leq n^d$.
\end{lemma}

Let $L_h(n)$ denote the complete $h$-partite $h$-uniform hypergraph with vertex classes of size $n$. We use the well known fact from Ramsey theory that any coloring of the edges of $L_h(M)$ for sufficiently large $M$ with fixed number of colors contains a monochromatic copy of $L_h(m)$. This follows, for example, from the Erd\H{o}s Box theorem \cite{erdos_box}, which states that if an $h$-partite $h$-uniform hypergraph $H$ with vertex classes of size  $M$ contains no copy of $L_h(m)$, then $H$ has at most $O_{h,m}(M^{h-1/m^{h-1}})$ edges. In particular, this shows that the most popular color class already contains the desired monochromatic subhypergraph, given $M$ is sufficiently large.

\begin{lemma}[$h$-partite Ramsey theorem]\label{lemma:Ramsey}
Given positive integers $h,m,q$, there exists $M=M(h,m,q)$ such that the following holds. Every $q$ coloring of the edges of $L_h(M)$ contains a monochromatic copy of $L_h(m)$.
\end{lemma}

Finally, we use the following well-known supersaturation  result  of Erd\H{o}s and Simonovits \cite{ES} about hypergraphs.

\begin{lemma}[Hypergraph supersaturation]\label{lemma:saturation}
Let $h,m$ be positive integers. For every $\alpha>0$ there exists $\gamma=\gamma(h,m,\alpha)>0$ such that if $n$ is sufficiently large, the following holds. Let $H$ be an $h$-uniform hypergraph on $n$ vertices with at least $\alpha \binom{n}{h}$ edges. Then $H$ contains at least $\gamma n^{hm}$ copies of $L_h(m)$.
\end{lemma}

\section{QFH for \texorpdfstring{$(2d)$}{(2d)}-tuples}

In this section, we prove Theorem \ref{thm:QFH_2d}. We highlight that our proof is quite short and also fairly elementary. It is partially inspired by the proof of the Colorful Helly theorem found in \cite{Barany82}.

\begin{proof}[Proof of Theorem \ref{thm:QFH_2d}]
    Let $w=d^{-cd}$, where $c>0$ is given by Lemma \ref{lemma:quantitative}. We  prove the following slightly stronger statement. For every positive integer $d$, and every $\varepsilon>0$ there exist $c_0=c_0(d,\varepsilon)>0$ such that the following holds. Let $\cF$ be a family of $n$ convex sets in $\mathbb{R}^d$ such that at least $\alpha \binom{n}{2d}$ of the $(2d)$-tuples of members of $\cF$ have an intersection of volume at least $1$. Then there are at least $c_0\alpha n$ members of $\cF$ whose intersection has volume at least $(1-\varepsilon)w$.

    We may assume that $n\geq 4d$, as the case $n<4d$ automatically follows by taking $c_0$ sufficiently small. Furthermore, we may assume that every element of $\cF$ is bounded. For $t\in \mathbb{R}$, define the halfspace $H(t)=\{(x_1,\dots,x_d)\in \mathbb{R}^d: x_d\leq t\}$. Given a convex set $K$ of volume at least $1$, the \emph{covering halfspace} of $K$ is the halfspace  $H(t)$ with the unique $t$ such that $\vol(K \cap H(t)) = 1$. (In case $\vol(K)=1$, take the smallest such $t$.)

    Say that a $(2d)$-tuple $I$ of elements of $\cF$ is  \emph{good} if $\vol(\cap_{K \in I} K)\geq 1$. Given a good $(2d)$-tuple $I$, consider all $(2d-1)$-element subsets of $I$. Among these, assign to $I$ the $(2d-1)$-element subset $J\subset I$ for which the covering halfspace $H(t)$ of $\cap_{K \in J}K$ has the largest $t$ (in case of draw, assign one of the largest arbitrarily). As there are at least $\alpha \binom{n}{2d}$ good $(2d)$-tuples, there exists a $(2d-1)$-tuple $J_0$ that is assigned to at least 
    $$\frac{\alpha\binom{n}{2d}}{\binom{n}{2d-1}}=\frac{\alpha}{2d} (n-2d+1)\geq \frac{\alpha}{4d}n:=m$$ pieces of good $(2d)$-tuples. Let $H_0$ be the covering halfspace of $J_0$, and let $\cF_0\subset \cF$ be a set of $m$ elements such that  $I=J_0\cup\{K\}$ is good and $J_0$ is assigned to $I$ for every $K\in \cF_0$.

    \begin{claim}
          $\vol(H_0 \cap K\cap \bigcap_{C \in J_0} C) \geq w$ for every $K\in \cF_0$.
    \end{claim}
    \begin{proof}
       The claim follows by applying the Quantitative Helly theorem (Lemma \ref{lemma:quantitative}) to the family $\mathcal{S}=\{H_0,K\}\cup J_0$. In order to do this, we need to show that any $(2d)$-tuple $\mathcal{S}_0$ of elements of $\mathcal{S}$ have an intersection of volume at least 1. We consider three cases:
        \begin{description}
            \item[Case 1.] $\mathcal{S}_0=\{H_0\}\cup J_0$.
            
            Then $\vol(H_0\cap \bigcap_{C\in J_0} C)=1$ as $H_0$ is a covering halfspace of $\bigcap_{C\in J_0} C$.

            \item[Case 2.] $\mathcal{S}_0=\{K\}\cup J_0$.

            Then $\vol(K\cap \bigcap_{C\in J_0} C)\geq 1$ follows as $\mathcal{S}_0$ is a good $(2d)$-tuple of $\cF$.

            \item[Case 3.] $\mathcal{S}_0=\{H_0,K\}\cup (J_0\setminus \{D\})$ for some $D\in J_0$.

            Let $H_1$ be the covering halfspace of $J_1=K\cup (J_0\setminus \{D\})$. As $J_0$ is assigned to $J_0\cup \{K\}$, we have $H_1\subset H_0$. Hence, $\vol(H_0\cap \bigcap_{C\in J_1}C)\geq \vol(H_1\cap \bigcap_{C\in J_1}C)=1$.
        \end{description}
    \end{proof}

    Let $L = H_0 \cap \bigcap_{C \in J_0} C$. Then $\vol (L) = 1$ and $\vol(K \cap L) \geq w$ for every  $K \in \cF_0$ by the previous claim. Also, after applying a volume preserving affine transformation and using Lemma \ref{lemma:ellipsoid}, we may assume that $L$ is contained in a ball $B$ of radius $d$. Let $\mathcal{B}$ be the set of convex bodies in $B$ of volume at least $w$, and let $\mathcal{M}$ be the set of convex bodies of volume $(1-\varepsilon)w$. For $M\in \mathcal{M}$, let $\mathcal{D}_{M}$ be the set of convex bodies in $\mathbb{R}^d$ that contain $M$ in their interior. Each $\mathcal{D}_{M}$ is open with respect to the Hausdorff distance, and $\mathcal{B}\subseteq \cup_{M\in \mathcal{M}}\mathcal{D}_M$. Note that $\mathcal{B}$ is compact, thus there is a finite $\mathcal{M}'\subset \mathcal{M}$ with $m':=|\mathcal{M}'|$ such that $\mathcal{B}\subseteq \cup_{M\in \mathcal{M}'}\mathcal{D}_M$. In particular, note that $m'$ only depends on $d$ and $\varepsilon$. Since $(K\cap L)\in \mathcal{B}$ for every $K\in \mathcal{F}_0$, there is an $M$ for which $(K\cap L)\in \mathcal{D}_M$ for at least $\frac{m}{m'}$ sets $K$. The intersection of these $\frac{m}{m'}$ members $K$ contain a convex body of volume $(1-\varepsilon)w$. Thus we are done by choosing $c_0=\frac{\alpha}{4dm'}$.
\end{proof}

Observe that we proved the following slightly stronger statement. Let $w=w(d)>0$ be the largest number with the property that if $\cF$ is a family of $2d+1$ convex sets in $\mathbb{R}^d$ such that any $2d$ members intersect in volume at least 1, then $\vol(\bigcap_{C\in \cF}C)\geq w$. Then our lower bound on the volume in  Theorem \ref{thm:QFH_2d} can be replaced with $(1-\varepsilon)w$ for any $\varepsilon>0$ (with the cost of decreasing $c_0$). It seems likely that $w$ is magnitudes larger than $d^{-O(d)}$, however, we will not pursue this direction.
 

\section{QFH for \texorpdfstring{$(d+1)$}{(d+1)}-tuples}

In this section, we present the proof of Theorem \ref{thm:QFH_d+1}. Our goal is to show that a positive fraction of $(d+1)$-tuples intersecting in volume 1 implies that a positive fraction of $(2d)$-tuples intersect in large volume, in which case we can apply Theorem \ref{thm:QFH_2d}. In order to prove this, we establish the following weak quantitative version of the Colorful Helly theorem. We show that if $\cF_1,\dots,\cF_{d+1}$ are sufficiently large families of convex sets such that every colorful selection $(A_1,\dots,A_{d+1})\in \cF_1\times\dots\times \cF_{d+1}$ intersect in  volume at least 1, then one of the families contains $2d$ members which intersect in large volume. Then, we use hypergraph supersaturation results to find many such $d+1$-tuples of families, and thus many $(2d)$-tuples intersecting in large volume. This latter argument is partially inspired by the work of B\'ar\'any and Matou\v{s}ek \cite{BM}, where they established a fractional version of Helly's theorem for convex lattice sets. 

In what follows, we prepare the proof of our weak quantitative Colorful Helly theorem. Say that a $(d+1)$-tuple of families $(\mathcal{F}_1,\dots,\cF_{d+1})$ of convex sets in $\mathbb{R}^d$ is \emph{$(m,\alpha,s,\varepsilon)$-rainbow} if
\begin{itemize}
    \item $|\cF_i|=m$ for $i\in [d+1]$,
    \item $\vol(A_1\cap\dots \cap A_{d+1})\geq \alpha$ for every  $(A_1,\dots,A_{d+1})\in \cF_1\times\dots\times \cF_{d+1}$,
    \item for every $i\in [d+1]$ and  distinct $B_1,\dots,B_{s}\in \cF_i$, $\vol(B_1\cap\dots\cap B_s)\leq \varepsilon.$
\end{itemize}

\begin{lemma}\label{lemma:sparsify}
Given positive integers $d,s,m$, if $0<\varepsilon\leq \varepsilon_0(d,s,m)$ is sufficiently small, then there exist $\alpha=\alpha(d,s,m)$ and $M=M(d,s,m)$ such that the following holds. Let $(\mathcal{F}_1,\dots,\cF_{d+1})$ be $(M,1,s,\varepsilon)$-rainbow. Then there exists a family $\cF^*$ of convex sets in $\mathbb{R}^d$ such that no $s$ members of $\cF^{*}$ have a common intersection, and there exist subfamilies $\cF_i'\subset \cF_i$ for $i=2,\dots,d+1$ such that $(\cF^*,\cF_2',\dots,\cF_{d+1}')$ is $(m,\alpha,s,\varepsilon)$-rainbow.
\end{lemma}

\begin{proof}
Let $\cF$ be an arbitrary $m$ element subset of $\cF_1$. For every $s$-tuple $\cB=\{B_1,\dots,B_s\}$ of elements of $\cF$, let $S_{\cB}$ be a smallest volume simplex containing $B_1\cap\dots\cap B_s$ (in case this intersection is empty, set $S_{\cB}=\emptyset$). By Lemma \ref{lemma:smallest_simplex}, there exists $c_1=c_1(d)>0$ such that $\vol(S_{\cB})\leq c_1\varepsilon$. Let $S=\bigcup_{\cB} S_{\cB}$, then every point of $\mathbb{R}^d\setminus S$ is covered by less than $s$ elements of $\cF$, and $\vol(S)\leq c_1 \varepsilon \binom{m}{s}$. Set $\varepsilon_0(d,s,m)=1/(2c_1\binom{m}{s})$, so that $\vol(S)\leq 1/2$ for every $\varepsilon<\varepsilon_0$. The face-hyperplanes of the simplices cut the space into at most $r=r(d,s,m):=((d+1)\binom{m}{s})^d$ convex regions by Lemma \ref{lemma:hyperplane_cutting}, let $R_1,\dots,R_q$ be the interiors of those regions that are disjoint from the interior of $S$.

Given a $d$-tuple $(A_2,\dots,A_{d+1})\in \cF_2\times\dots\times\cF_{d+1}$, we assign a color from $[q]^m$ to it as follows. For every $B\in \cF$, we have $\vol((B\cap A_2\dots\cap A_{d+1})\setminus S)\geq 1-\vol(S)\geq 1/2$, hence there exists some $j_B\in [q]$ such that $\vol(R_{j_B}\cap B\cap A_2\dots\cap A_{d+1})\geq 1/(2q)\geq 1/(2r)=:\alpha$. Assign the color $(j_B)_{B\in \cF}$ to $(A_2,\dots,A_{d+1})$. By Lemma \ref{lemma:Ramsey} applied with uniformity $d$ and number of colors $q^m$, we deduce  that if $M=M(d,s,m)$ is sufficiently large, then we can find a color $(j_B)_{B\in \cF}$ and $\cF_i'\subset \cF_i$ of size $m$ for every $i\in \{2,\dots,d+1\}$ such that each $d$-tuple of $\cF_2'\times\dots\times \cF_{d+1}'$ is of color $(j_B)_{B\in \cF}$. Setting $\cF^{*}=\{B\cap R_{j_B}:B\in \cF\}$, we get that $(\cF^*,\cF_2',\dots,\cF_{d+1}')$ is  $(m,\alpha,s,\varepsilon)$-rainbow and no $s$ elements of $\cF^{*}$ have a common intersection.  
\end{proof}

\begin{lemma}[Weak Quantitative Colorful Helly theorem]\label{lemma:QCFH}
Given positive integers $d,s$, there exists $N=N(d,s)$ and $\delta=\delta(d,s)>0$ such that the following holds. If $(\cF_1,\dots,\cF_{d+1})$ is $(m,1,s,\delta)$-rainbow, then $m<N$.
\end{lemma}

\begin{proof}
Let $M,\alpha,\varepsilon_0$ be the functions given by Lemma \ref{lemma:sparsify}. Let $N_{d+1}=s$, and for $i=d,d-1,\dots,0$, let $N_i=M(d,s,N_{i+1})$. Furthermore, let $\beta_0=1$, and for $i=1,\dots,d+1$, define $\beta_i:=\beta_{i-1}\cdot\alpha(d,s,N_{i})$. Finally, let $\delta=\min_{i\in \{0,\dots,d\}}\beta_{i}\cdot \varepsilon_0(d,s,N_{i+1})$. We show that $\delta$ and $N=N_0$ suffices.

Assume to the contrary that there exists an $(N,1,s,\delta)$-rainbow $(d+1)$-tuple $(\cF_1,\dots,\cF_{d+1})$. Then, for $i=0,\dots,d+1$, we construct a $(d+1)$-tuple $(\cF_1^{(i)},\dots,\cF_{d+1}^{(i)})$ that is $(N_{i},\beta_i,s,\delta)$-rainbow and the $i$ families $\cF_1^{(i)},\dots,\cF_i^{(i)}$ have the property that the intersection of any $s$ of their elements is empty. Call this latter property as \emph{property $(i)$}. We do this by induction on $i$: for $i=0$, taking $\cF^{(0)}_j=\cF_j$ for $j\in [d+1]$ suffices. Having already defined $(\cF_1^{(i)},\dots,\cF_{d+1}^{(i)})$ that is $(N_{i},\beta_i,s,\delta)$-rainbow and has property $(i)$, we proceed as follows. As $N_i=M(d,s,N_{i+1})$ and $\delta\leq \beta_i\cdot \varepsilon_{0}(d,s,N_{i+1})$, there exists a family of convex sets $\cF_{i+1}^{(i+1)}$ such that no $s$ elements of  $\cF_{i+1}^{(i+1)}$ have a common intersection, and there exist $\cF^{(i+1)}_j\subset \cF^{i}_j$ for $j\in [d+1]\setminus \{i+1\}$ such that $(\cF_1^{(i+1)},\dots,\cF_{d+1}^{(i+1)})$ is $(N_{i+1},\beta_{i+1},s,\delta)$-rainbow. But then this $(d+1)$-tuple also has property $(i+1)$, finishing our induction step.

But observe that the family $(\cF_1^{(d+1)},\dots,\cF_{d+1}^{(d+1)})$ is $(s,\beta_{d+1},s,\delta)$-rainbow, which in particular implies that $A_1\cap\dots\cap A_{d+1}\neq \emptyset$ for any $(A_1,\dots,A_{d+1})\in \cF_1^{(d+1)}\times\dots\times\cF_{d+1}^{(d+1)}$. But this $(d+1)$-tuple also has property $(d+1)$, which means that none of the families $\cF_1^{(d+1)},\dots,\cF_{d+1}^{(d+1)}$ is intersecting. This contradicts the Colorful Helly theorem (Lemma \ref{lemma:colorful}), finishing the proof.
\end{proof}

Now everything is set to prove our main theorem.

\begin{proof}[Proof of Theorem \ref{thm:QFH_d+1}]
We may assume that $n\geq n_0$ is sufficiently large with respect to $d$ and $\alpha$, as we may always take $\beta$ to be less than $1/n_0$, which immediately implies the theorem for $n\leq n_0$. Recall that $L_h(m)$ is the complete $h$-partite $h$-uniform hypergraph with vertex classes of size $m$.

Let $H$ be the $(d+1)$-uniform hypergraph on vertex set $\cF$, in which a $(d+1)$-element subset $\{A_1,\dots,A_{d+1}\}$ forms an edge if $\vol(A_1\cap\dots\cap A_{d+1})\geq 1$. Then $H$ has $n$ vertices and at least $\alpha\binom{n}{d+1}$ edges. Applying Lemma \ref{lemma:QCFH} with $s=2d$, there exists $N=N(d,2d)$ and $\delta=\delta(d,2d)$ such that if $Q$ is a copy of $L_{d+1}(N)$ in $H$, then $V(Q)$ contains a subset $\{B_1,\dots,B_{2d}\}$ of size $2d$ such that $\vol(B_1\cap\dots\cap B_{2d})\geq \delta$. 

Let $H'$ be the $(2d)$-uniform hypergraph on vertex set $\cF$ in which $\{B_1,\dots,B_{2d}\}$ is an edge if $\vol(B_1\cap\dots\cap B_{2d})\geq \delta$. By Lemma \ref{lemma:saturation}, there exists $\gamma=\gamma(2d,N,\alpha)$ such that if $n$ is sufficiently large, then $H$ contains at least $\gamma n^{(d+1)N}$ copies of $L_{d+1}(N)$. By the previous paragraph, the vertex set of each such copy contains an edge of $H'$. As every edge of $H'$ can be contained in at most $n^{(d+1)N-2d}$ copies of $L_{d+1}(N)$, we deduce that $H'$ has at least $\gamma n^{2d}>\gamma\binom{n}{2d}$ edges. By Theorem \ref{thm:QFH_2d}, there exists $c_0=c_0(d)>0$ and $c_1=c_1(d)>0$ such that at least $c_0\gamma n$ members of $\cF$ have a common intersection of volume at least $\delta\cdot d^{-c_1 d}$. Hence, setting $\beta=c_0 \gamma$ and $v(d)=\delta\cdot d^{-c_1 d}$ finishes our proof.
\end{proof}

\section{Quantitative \texorpdfstring{$(p,q)$}{(p,q)}-theorems}\label{sec:pq}

The proof of Theorem~\ref{thm:pqsmallp} is based on the proof of \cite[Theorem 1.7]{JN}, which is a quantitative $(p,3d+1)$-theorem. It will be convenient to measure the volume of convex sets by the volume of the maximum volume inscribed ellipsoid. Let $v' = v(d)$ be the bound on the volume from Theorem~\ref{thm:QFH_d+1}, and let $v = v'd^{-d}$. Note that if a convex set is of volume $v'$, then it contains an ellipsoid of volume $v$ by Lemma~\ref{lemma:ellipsoid}. The following is a quantitative analog of the Weak Epsilon Net theorem of Alon, B\'{a}r\'{a}ny, F\"{u}redi and Kleitman \cite{ABFK92}, and can be found in \cite{JN}.

\begin{theorem}[Theorem 5.4 in \cite{JN}]\label{thm:weakepsilonnet}
	For every $d \geq 1$ there exists a function $f: (0,1] \rightarrow \Re$ with the following property. Let $\cF$ be a finite family of convex sets in $\Red$, and let $\cE$ be a finite family of volume 1 ellipsoids. For every $\varepsilon\in(0,1]$, if $w:\cE\rightarrow[0,1]$ is a weight function such that $\sum_{E\in\cE} w(E)=1$ and $$\sum_{E\in\cE, E\subseteq C} w(E)\geq \varepsilon$$ for all $C\in\cF$, 
	then there is a family $\cS$ of ellipsoids of volume 1 such that each $C \in \cF$ contains a member of $\cS$, and $\cS$ is of size at most $f(\varepsilon)$.
\end{theorem}

This theorem essentially states that we can bound the quantitative transversal number by the quantitative fractional transversal number, which we define below.

\begin{definition}[Quantitative fractional $v$-transversal number]
	For a set $S \subseteq \Red$, let $E_v(S)$ be the set of volume $v$ ellipsoids contained in $S$.
	Let $\C$ be a family of subsets of $\Red$ and let $\varphi: E_v(\Red) \rightarrow [0,1]$ be a function with finite support. We say that $\varphi$ is a \emph{quantitative fractional $v$-transversal} for $\C$, if $\sum_{E \in E_v(C)} \varphi(E) \geq 1$ for all $C \in \C$.
	The \emph{quantitative fractional $v$-transversal number} of $\C$ is the infimum of $\sum_{E \in E_v(\Red)}\varphi(E)$ over all quantitative fractional $v$-transversals $\varphi$ of $\C$, and it is denoted by $\tau^{*}_v(\C)$.
\end{definition}

To prove our quantitative $(p,d+1)$-theorem, we bound the quantitative fractional $v$-transversal number of families satisfying the condition of Theorem~\ref{thm:pqsmallp}. The following is an analog of \cite[Lemma 7.1]{JN}, which itself is an adaptation of the argument used by Alon and Kleitman in \cite{AK92}.

\begin{lemma}\label{lemma:boundedvtransversal}
	For every positive integer $d$ and $p \geq d+1$, there exists a $h = h(p,d)>0$ with the following property. Let $\cF$ be a finite family of convex sets in $\Red$, each containing an ellipsoid of volume 1, and assume that among any $p$ members of $\cF$, there exist $d+1$ whose intersection contains an ellipsoid of volume 1. Then $\tau^{*}_v(\cF)<h$.
\end{lemma}

\begin{proof}
    Let us define the dual linear program of $\tau_v^{*}(\C)$, which we call \emph{quantitative fractional $v$-matching number}. Let $\C$ be a finite family of convex sets in $\Red$, and let $m: \C \rightarrow [0,1]$ be a function. We say that $m$ is a \emph{quantitative fractional $v$-matching} for $\C$, if for every ellipsoid $E \subset \Red$ with volume $v$, 
    $$\sum_{C\in \C: E \subset C}m(C)\leq 1.$$
    The \emph{quantitative fractional $v$-matching number} of $\C$ is the supremum of $\sum_{C \in \C}m(C)$ over all quantitative fractional $v$-matchings of $\C$, and it is denoted by $\nu_v^*(\C)$. By the duality of linear programming, we have $\nu_v^{*}(\C)=\tau_v^{*}(\C)$.
	
	As $\nu_v^{*}(\cF)$ is the solution of a linear program with integer coefficients, there is an optimal quantitative fractional $v$-matching taking only rational values. Let $m$ be such a quantitative fractional $v$-matching and write $m(C) = \frac{\tilde{m}(C)}{D}$ for every $C\in\cF$, where $\tilde{m}(C)$ and $D$ are integers.
	
	Let $\tilde{\cF} = \{C_1, \ldots, C_N\}$ be the multiset that contains $\tilde{m}(C)$ copies of each $C \in \cF$. Thus, we have
	\begin{equation}\label{eq:nuv}
	 N:=\sum_{C\in\cF}\tilde{m}(C)=D\cdot \nu_v^*(\cF).
	\end{equation}
	Taking some sets with multiplicities does not change the quantitative fractional matching number, so $\nu_v^*(\tilde{\cF}) = \nu_v^*(\cF)$.  By our assumption on $\cF$, among any $p$ members of $\cF$, there are $d+1$ whose intersection contains an ellipsoid of volume $1$. Therefore, among any $\tilde{p} = d(p-1) + 1$ members of $\tilde{\cF}$, there are $d+1$ whose intersection contains an ellipsoid with volume $1$. This is true as a $\tilde{p}$ element multiset from $\tilde{\cF}$ either contains $d+1$ copies of the same set, or $p$ different sets from $\cF$. Thus, for every $I \in \binom{[N]}{\tilde{p}}$, there is a subset $J \subset I$ such that $|J| = d+1$ and $\cap_{j \in J} C_j$ contains an ellipsoid of volume 1. Hence, there are at least
	\[
	\frac{\binom{N}{\tilde{p}}}{\binom{N-d-1}{\tilde{p}-d-1}}=
	\alpha \binom{N}{d+1}
	\]
	$d+1$-tuples from $\tilde{\cF}$ whose intersection contains an ellipsoid of volume 1, where $\alpha=1/(\tilde{p}\dots(\tilde{p}-d))$. We can apply Theorem~\ref{thm:QFH_d+1} and conclude that we can find $\beta N$ sets in $\tilde{\cF}$ whose intersection is of volume $v'$ for some $\beta=\beta(d,\alpha)>0$, thus it contains an ellipsoid of volume $v$.
	
	On the other hand, for any ellipsoid $E$ of volume $v$, we have
	\[
	 \sum_{C\in\cF : E\subseteq C} \frac{\tilde{m}(C)}{D}=\frac{1}{D}\#\{C\in\tilde{\cF}: E\subseteq C\}\leq 1.
	\]
    Thus, by \eqref{eq:nuv}, no ellipsoid of volume $v$ can be in more than $D=\frac{N}{\nu_v^*(\cF)}$ of the sets from $\tilde{\cF}$. Hence, $\tau_v^{*}(\cF)=\nu_v^*(\cF) \leq \frac{1}{\beta}$, completing the proof.
\end{proof}

\begin{proof}[Proof of Theorem ~\ref{thm:pqsmallp}]
    If among any $p$ members of $\cF$ there are $d+1$ whose intersection contains an ellipsoid of volume $1$, then it follows from Lemma~\ref{lemma:boundedvtransversal} that $\tau_v^{*}(\cF)\leq h$ for some $h=h(p,d)$. Let $\varphi$ be a quantitative fractional $v$-transversal satisfying $$\sum_{E\in E_v(\Red)}\varphi(E)\leq h,$$ and for every ellipsoid $E$ of volume $v$, let $w(E) = \frac{\varphi(E)}{h}.$ Since for every $C \in \cF$, the inequality $\sum_{E \subset C}w(E) \geq \frac{1}{h}$ is satisfied, we can apply Theorem~\ref{thm:weakepsilonnet} with $\varepsilon=1/h$ to conclude that there is a family $\cS$ of at most $f\left(\frac{1}{h}\right)$ ellipsoids with volume $v$ such that every element of $\cF$ contains at least one element of $\cS$. Hence, our theorem follows with $r = |\cS| = f\left(\frac{1}{h}\right)$.
\end{proof}

\section{Diameter Helly results}\label{sect:diameter}

In order to prove Theorems~\ref{thm:QFHdiameter_d+1} and \ref{thm:diamter_pq}, one only needs to  modify a few geometric elements  of Theorems~\ref{thm:QFH_d+1} and \ref{thm:pqsmallp}. In the following, we describe how to replace the (volumetric) geometric ingredients of the proofs by analogous results about the diameter. The rest of the proofs of Theorems~\ref{thm:QFHdiameter_d+1} and \ref{thm:diamter_pq} can be obtained by replacing every occurence of the word ``volume'' by ``diameter'' in the proofs of Theorems~\ref{thm:QFH_d+1} and \ref{thm:pqsmallp}.

In the proof of the diameter version of Lemma~\ref{lemma:sparsify} two steps need to be modified. First, the following result of Gale \cite{Gale} can be used in place of Lemma~\ref{lemma:smallest_simplex}.

\begin{lemma}[Gale \cite{Gale}]
    Let $K$ be a convex set in $\mathbb{R}^d$. Then there exists a simplex of diameter at most $(cd)\cdot \diam(K)$ containing $K$ for some absolute constant $c>0$.
\end{lemma}

Second, we need the following property of the diameter.

\begin{claim}
    Let $C$ be a convex set of diameter at least one and let $S$ be a union of at most $\binom{m}{s}$ sets of diameter at most $c_1\varepsilon$. If $\varepsilon \leq \varepsilon_0(d,s,m)$, then $\diam(C \setminus S) \geq 1/2$.
\end{claim}

In the proof of Theorem~\ref{thm:QFH_d+1} we use a reduction to Theorem~\ref{thm:QFH_2d}. In the diameter case, Theorem~\ref{thm:QFHdiameter_d+1} can be reduced to Theorem 4.1 in \cite{Soberon}, an analogous result by Sober\'on for the diameter.

In the proofs of the $(p,d+1)$-theorems, the existence of weak epsilon nets is used. In the diameter proof we can replace Theorem~\ref{thm:weakepsilonnet} by Theorem 4.4 in \cite{Soberon}, again an analogous result by Sober\'on for the diameter

\section*{Acknowledgement}
We would like to thank Travis Dillon for pointing out that our proof works in the diameter case, M\'arton Nasz\'odi for fruitful discussions, Zach Hunter for pointing out some mistakes, and the Oberwolfach Workshop 2404 in Discrete Geometry for providing an inspiring atmosphere to collaborate.

\end{document}